\documentclass[a4paper,12pt,twoside]{article}
\usepackage[english]{babel}
\usepackage{amsmath,amsfonts,amsthm,amssymb,graphicx,color,multicol}
\usepackage{tabularx}
\usepackage{dsfont}
\usepackage{subfigure}
\usepackage[top=1.5in,bottom=1in,left=1.5in,right=1.5in]{geometry}
\usepackage{url}
\usepackage{fancyhdr}
\usepackage[latin1]{inputenc}
\newcolumntype{C}[1]{>{\centering\let\newline\\\arraybackslash\hspace{0pt}}m{#1}}

\makeatletter
\newcommand{\rmnum}[1]{\romannumeral #1}
\newcommand{\Rmnum}[1]{\expandafter\@slowromancap\romannumeral #1@}

\def\e{\mbox{e}}
\def\E{\mathbb{E}}
\def\P{\mathbb{P}}

\def\R{\mathbb{R}}
\def\In{\infty}
\theoremstyle{definition}
\newtheorem{theorem}{\bf Theorem}[section]

\newtheorem{lemma}[theorem]{Lemma}

\theoremstyle{definition}

\title{Upper bound for the tail functions of the growth rate for supercritical branching processes in random environment}
\date{}
\author{Yinna Ye\footnote {Department of Statistics and Actuarial Science, School of Science, Xi'an Jiaotong-Liverpool University, Suzhou 215123, China. E-mail: yinna.ye@xjtlu.edu.cn}}
\begin{document}
\maketitle

{\bf Abstract.} Suppose that $(Z_n)_{n\geq0}$ is a supercritical branching process in independent and identically distributed random environment. The right tail function of the scaled growth rate for $(Z_n)_{n\geq0}$ is studied. The upper bounds for $\displaystyle\P\left[\frac{\log Z_n}{Mn}-\mu\geq x\right]$ for any $x\geq3$ are obtained, by applying an extension of the Hoeffding type inequalities.\\
\vspace{0.3cm}

\vspace{0.3cm}


\section{Introduction and setting}
The branching process in random environment (BPRE) has become a hot topic since 1970's. One of the best-known works since then are those developed by Athreya and Karlin in 1971 (\cite{Ath1}, \cite{Ath2}), based on the Smith and Wilkinson's model (\cite{Smi}) where the random environment process is supposed to be independent and identically distributed (i.i.d.) random variables (r.v.s), they extend their model to more general situations of the random environment like e.g. stationary and metrically transitive random process or Markov chain. They found 'extinction or explosion result' for BPRE and delimited complete criteria for certainty or noncertainty of extinction and clarified the three classes: subcritical, critical and supercritical BPRE for the first time. Then Tanny \cite{Tan} improved the conditions of 'extinction or explosion result' and studied the rate of growth of the BPRE especially when the branching process (BP) is supercritical. Later on, based on these results, a series of papers appeared and studied especially BPs in i.i.d. random environment. Their focus are mainly in two directions: one is on the asymptotic of the survival probability for subcritical BPs, the typical papers are \cite{Koz}, \cite{Gei1}, \cite{Gui} \cite{Gei2} and \cite{Emile}; another one is on the large deviations of supercritical BPs, the typical papers are \cite{Boi}, \cite{Chu}, \cite{Nak} and \cite{Gra}.\\

However, so far the Hoeffding type inequalities for BPRE has not been studied yet in the literature. If $(Y_i)_{1\leq i\leq n}$ is a sequence of centered ($\E(Y_i)=0$) r.v.s with finite variance. The Hoeffding type inequalities provide upper bounds for the right tail function of $\sum_{i=1}^nY_i$. Particularly, if $(Y_i)_{1\leq i\leq n}$ are independent r.v.s satisfying $Y_i\leq 1$, the classical Hoeffding inequality, initially developed by Bennett \cite{Ben} and Hoeffding \cite{Hoef}, is in the following form: for any $t>0$,
 $$\P\left(\sum_{i=1}^nY_i\geq nt\right)\leq\left(\frac{\tau^2}{t+\tau^2}\right)^{n(t+\tau^2)}\e^{nt},$$
 where $\displaystyle\tau^2=\frac{1}{n}\sum_{i=1}^n\E(Y^2_i)$. More recently, Fan et al. \cite{Fan} improved the classical upper bound and extended their results to the case when  $\left(\sum_{i=1}^n Y_i\right)_{n}$
  is a supermatingale with differences bounded above.\\

The aim of the present paper is to study the right tail function of the growth rate in the form $\displaystyle\P\left[\frac{\log Z_n}{Mn}-\mu\geq x\right]$
for supercritical BPs in i.i.d. random environment and give an upper bound. More precisely, we try to see under the application of Hoeffding type inequalities developed by Fan \textit{et al.} \cite{Fan} recently, how large the range of $y\in\R^+$ may attain for validation of the potential convergence in probability of the sample mean $\displaystyle\frac{\log Z_n}{n}$ as follows
$$\P\left(\left|\frac{\log Z_n}{n}-\mu\right|\geq y\right)\stackrel{n\rightarrow+\In}{\longrightarrow}0.$$
The main result in the present paper (cf. Theorem \ref{Thm1}) shows that under certain conditions, when $x\geq3$,
$$\P\left[\frac{\log Z_n}{Mn}-\mu\geq x\right]\stackrel{n\rightarrow+\In}{\longrightarrow}0.$$
To our best knowledge, unfortunately, when $x$ is small (for instance, $0\leq x<3$), Hoeffding type inequalities seem not effective to obtain the above convergence. Nevertheless, the conclusion above is natural and reasonable, since $\log Z_n$ is not a sum of i.i.d random variables anymore in general.\\

The paper is organized as follows. In the rest of this section, the BPs in i.i.d. random environment and the assumptions will be introduced. The mail result Theorem \ref{Thm1} and its proof are given in Section \ref{Sec2}. In Section \ref{Sec3}, an example of the binary branching in i.i.d. random environment will be discussed to show the feasibility. \\

Consider a branching process (BP), denoted by $(Z_n)_{n\geq0}$, evolving in i.i.d. random environment $\overline{\xi}=(\xi_0,\xi_1,\cdots)$ with common distribution $\nu$. The BPs in i.i.d. environment can be defined as follows. For instance, one may also refer to \cite{Chu} and \cite{Gra} for the definition. Let $Z_n$ denote the number of individuals at the $n$th generation in a family tree and $Z_n$ satisfies the following recursive form:
 $$Z_0\equiv1,\quad Z_{n+1}=\sum_{i=1}^{Z_n}N_{n,i},\text{ for }n\geq0,$$
  where $N_{n,i}$ represents the family size of the $i$th father individual located in the $n$th generation. Given the environment $\xi_n$ in the $n$th generation, $(N_{n,i})_{i\geq1}$ is a sequence of i.i.d. r.v.s with  the conditional probability mass function (p.m.f.) $(p_k(\xi_n))_{k\geq0}$. Because the one-to-one correspondence between p.m.f. and probability distribution, sometimes people also call $(p_k(\xi_n))_{k\geq0}$ the { (conditional) offspring distribution} in the $(n+1)$th generation. Suppose that given the { environment sequence} $\overline{\xi}$, every family size $(N_{n,i})_{n\geq0}$ are conditionally independent from generation to generation. Let
$$m_n:=\sum_{k=0}^{+\In}kp_k(\xi_n), \quad\text{for }n\geq0;$$
$$\Pi_n:=m_0\cdots m_{n-1}, \text{ for }n\geq1\text{ and } \Pi_0=1;$$
$$S_n:=\log\Pi_n=\sum_{i=1}^nX_i, \text{ for } n\geq1,$$
where for $i\geq1$, $X_i:=\log m_{i-1}$ is the log-conditional mean of the family size of a father individual located in the $(i-1)$th generation given the environment $\xi_{i-1}$ in that generation. Then we have
\begin{equation}\label{eq1}
\log Z_n=S_n+\log W_n,
\end{equation}
from which it can be seen that $\log Z_n$ is decomposed into two components: $S_n$ and $\log W_n$. It is well known (see also \cite{Chu} and \cite{Gra} for instance) that $(S_n)_{n\geq1}$ is a random walk with i.i.d increments and $(W_n)_{n\geq1}$ is a non-negative martingale under both quenched and annealed laws, denoted by $\P_{\overline{\xi}}$ and $\P$ respectively, with respect to (w.r.t) the filtration $(\mathcal{F}_n)_{n\geq1}$, where $\mathcal{F}_n$ is given by
$$\mathcal{F}_n=\sigma(\overline{\xi},N_{k,i},0\leq k\leq n-1,\, i=1,2,\ldots).$$
Moreover,
$$W=\lim_{n\rightarrow+\In}W_n\quad \text{exists }\P- a.s.$$
and
$$\E(W)\leq1.$$
Let $\mu=\E(X_1)$ and $\sigma^2=\mbox{Var}(X_1)$. Throughout the paper, we will use $\E_{\overline{\xi}}$, $\E$ and $\nu(\cdot)$ to present the expectations w.r.t the probability distributions $\P_{\overline{\xi}}$, $\P$ and $\nu$ respectively. We denote by $C$ an absolute constant whose value may differ from line to line. And we assume the following basic assumptions:
\begin{enumerate}
 \item[1)] $0<\mu<+\In$ and $\E|\log(1-p_0(\xi_0))|<+\In,$\\
 which implies the population size tends to $+\In$ with positive probability.
  \item[2)] $0<\sigma^2<+\In$,\\
  which implies that $\P(Z_1=1)=\E(p_1(\xi_0))<1$.
 \item[3)] $\displaystyle \E\left(\frac{Z_1\log^+{Z_1}}{m_0}\right)<+\In$,\\
 which is a necessary and sufficient condition (by Theorem 2 in \cite{Tan}) to imply that $W_n\rightarrow W$ in $\mathbb{L}^1$, as $n\rightarrow+\In$; and
 $$\P(W>0)=\P(Z_n\rightarrow+\In)=\lim_{n\rightarrow+\In}\P(Z_n>0)>0.$$
\end{enumerate}
Furthermore, assume that
$$p_0=0\quad \P- a.s.,$$
which implies $Z_n\longrightarrow+\In$ and $W>0$ $\P- a.s.$\\

We need the following two more hypothesis:
\begin{description}
 \item[H1)] Suppose $\exists M>0$, s.t. $\forall i\geq1$,
 $$\frac{X_i-\mu}{M}\leq1\quad \P- a.s.$$
in other words, $X_i$ is a bounded random variable almost surely.
 \item[H2)] There exist $p>1$ and $q>2$, s.t. $\displaystyle \E\left(\frac{Z_1}{m_0}\right)^p<+\In$ and $\E|\log m_0|^q<+\In$.
\end{description}

%
Throughout the paper, we denote by $C$ an absolute constant whose value may differ from line to line. We denote by $\mathds{1}$ the indicator function and $\overset{\mathcal{D}}{=}$ means the equality holds in distribution.\\

 Applying the Hoeffding type inequality in Theorem 2.1 Fan \emph{et al.} \cite{Fan} to $\displaystyle \frac{S_n-n\mu}{M}$, we have
\begin{theorem}[\cite{Fan}]\label{Fan}Under the assumption H1), for any $x\geq0$ and $n\geq1$,
$$\P\left(\frac{S_n-n\mu}{M}\geq x\right)\leq H_n(x,v_n),$$
where $v_n:=\sqrt{\sum_{i=1}^n\E\left(\frac{X_i-\mu}{M}\right)^2}=\frac{\sqrt{n}\sigma}{M}$ and $H_n$ is a function defined by
$$H_n(x,v):=\left[\left(\frac{v^2}{x+v^2}\right)^{x+v^2}\left(\frac{n}{n-x}\right)^{n-x}\right]^{\frac{n}{n+v^2}}\mathds{1}\{x\leq n\},$$
for any $x\geq0$ and $v>0$, with the convention $H_n(n,v)=1$.\\
Furthermore,
$$H_n(x,v)\leq \e^x \left(\frac{v^2}{x+v^2}\right)^{x+v^2}.$$
\end{theorem}

\begin{lemma}\label{Lem1}
For any $x\geq 0$ and $v>0$, the function $H_n(x,v)$ is non-increasing in $x$.
\end{lemma}
\begin{proof}
By the definition of $H_n$, when $0\leq x\leq n$,
$$\log H_n(x,v)=\frac{n}{n+v^2}\left\{(x+v^2)[2\log v-\log(x+v^2)]+(n-x)[\log n-\log(n-x)]\right\}.$$
 Taking partial derivative in both sides w.r.t $x$, we can obtain
 \begin{equation}\label{eqthree}
 \frac{\partial}{\partial x}H_n(x,v)=\frac{n}{n+v^2}\;H_n(x,v)\log\left[\frac{v^2(n-x)}{n(v^2+x)}\right]\leq 0,
 \end{equation}
 since for any $0\leq x\leq n$ and $v>0$, $H_n(x,v)>0$ and
 $$0<\frac{v^2(n-x)}{n(v^2+x)}\leq\frac{nv^2}{nv^2+nx}\leq1.$$
 And from the inequality above, the equality in (\ref{eqthree}) is attained only when $x=0$. In the case when $x>n$, we have $H_n(x,v)\equiv0$. We thus prove the desired result.
\end{proof}
\section{Main result and proof}\label{Sec2}
\begin{theorem}\label{Thm1}
Under the assumptions H1) and H2), there exist constants $C>0$ and $\delta\in(0,1)$ s.t. for any $x\geq3$ and $1\leq m\leq n$,
\begin{equation}\label{eqthm1}
\P\left(\frac{\log Z_n-n\mu}{nM}\geq x\right)\leq \frac{C\delta^m}{M\sqrt{n}}.
\end{equation}
\end{theorem}



The proof of Theorem \ref{Thm1} is inspired from that of Theorem 1.1 in \cite{Gra}.\\
\begin{proof}
Since from (\ref{eq1}) for any $n\geq1$,
$$\log Z_n-n\mu=(S_n-n\mu)+\log W_n,$$
we have
\begin{equation*}
\begin{split}
0\leq\P\left(\frac{\log Z_n-n\mu}{nM}\geq x\right)=\P\left(\frac{S_n-n\mu}{nM}\geq x\right)+&\P\left(\frac{\log Z_n-n\mu}{nM}\geq x,\frac{S_n-n\mu}{nM}<x\right)-\\
                                                        &\P\left(\frac{\log Z_n-n\mu}{nM}< x,\frac{S_n-n\mu}{nM}\geq x\right).
\end{split}
\end{equation*}
Therefore,
\begin{equation}\label{eq4}
\P\left(\frac{\log Z_n-n\mu}{nM}\geq x\right)\leq\P\left(\frac{S_n-n\mu}{nM}\geq x\right)+\P\left(\frac{\log Z_n-n\mu}{nM}\geq x,\frac{S_n-n\mu}{nM}<x\right).
\end{equation}
The upper bound for the first term in the right-hand side of the inequality above can be handled as below by Theorem \ref{Fan}, for $x\geq1$,
$$0\leq\P\left(\frac{S_n-n\mu}{nM}\geq x\right)=\P\left(\sum_{i=1}^n\frac{X_i-\mu}{M}\geq nx\right)\leq H_n\left(nx,\frac{\sqrt{n}\sigma}{M}\right)\equiv0,$$
since $nx\geq n$, for $x\geq1$.\\

To handle the second term, we need to study the joint distribution of $\left(\frac{\log Z_n-n\mu}{nM},\frac{S_n-n\mu}{nM}\right)$. The upper bound for the second term is obtained in the Lemma below. And the inequality (\ref{eqthm1}) is therefore proved.
\end{proof}
\begin{lemma}\label{Lem2}
Under assumptions H1) and H2), there exist constants $C>0$, $\delta\in(0,1)$, s.t. for $\displaystyle 1\leq m\leq n$ and $x\geq3$,
 $$\displaystyle \P\left(\frac{\log Z_n-n\mu}{nM}\geq x,\frac{S_n-n\mu}{nM}<x\right)\leq\frac{C\delta^m}{M\sqrt{n}}.$$
\end{lemma}
\begin{proof}
\begin{equation*}\label{eq5}
\begin{split}
&\P\left(\frac{\log Z_n-n\mu}{nM}\geq x,\frac{S_n-n\mu}{nM}<x\right)\\
=&\P\left[\frac{1}{n}\sum_{i=1}^n\left(\frac{X_i-\mu}{M}\right)+\frac{\log W_n}{nM}\geq x,\frac{1}{n}\sum_{i=1}^n\left(\frac{X_i-\mu}{M}\right)<x\right]\\
\leq& \P\left[\frac{1}{n}\sum_{i=1}^n\left(\frac{X_i-\mu}{M}\right)+\frac{\log W_n}{nM}\geq x-\frac{1}{\sqrt{n}}, \frac{1}{n}\sum_{i=1}^n\left(\frac{X_1-\mu}{M}\right)<x\right]+\\
&\qquad\qquad\qquad\qquad\qquad\qquad\qquad\qquad\P\left(\frac{|\log W_n-\log W_m|}{nM}>\frac{1}{\sqrt{n}}\right)\\
:=&I_1+I_2.
\end{split}
\end{equation*}
Let us study $I_1$ and $I_2$ respectively.
\begin{equation*}
\begin{split}
I_1&=\P\left[\frac{1}{n}\left(\sum_{i=1}^m\frac{X_i-\mu}{M}+\sum_{i=m+1}^n\frac{X_i-\mu}{M}\right)+\frac{\log W_n}{nM}\geq x-\frac{1}{\sqrt{n}},\frac{1}{n}\left(\sum_{i=1}^m\frac{X_i-\mu}{M}+\sum_{i=m+1}^n\frac{X_i-\mu}{M}\right)<x\right]\\
   &=\int\P\left[\frac{1}{n}\sum_{i=m+1}^n\left(\frac{X_i-\mu}{M}\right)+s+t\geq x-\frac{1}{\sqrt{n}},s+\frac{1}{n}\sum_{i=m+1}^n\left(\frac{X_i-\mu}{M}\right)<x \right]\,\nu_m(\mbox{d} s,\mbox{d} t),
\end{split}
\end{equation*}
where $\nu_m(\mbox{d} s,\mbox{d} t)$ is the joint distribution of $\displaystyle \left(\frac{1}{n}\sum_{i=1}^m\frac{X_i-\mu}{M},\frac{\log W_m}{nM}\right)$. By conditioning on $\displaystyle \left(\frac{1}{n}\sum_{i=1}^m\frac{X_i-\mu}{M},\frac{\log W_m}{nM}\right)$ and the independence between $\displaystyle \left(\frac{1}{n}\sum_{i=m+1}^n\frac{X_i-\mu}{M}\right)$ and $\displaystyle \left(\frac{1}{n}\sum_{i=1}^m\frac{X_i-\mu}{M},\frac{\log W_m}{nM}\right)$, for $0\leq m\leq n-1$, the equality above becomes
\begin{equation}\label{eq6}
I_1=\int\mathds{1}_{\left\{t\leq\frac{1}{n}\right\}}\left[\P\left(\frac{1}{n}\sum_{i=m+1}^n\frac{X_i-\mu}{M}\geq x-s-t-\frac{1}{\sqrt{n}}\right)-\P\left(\frac{1}{n}\sum_{i=m+1}^n\frac{X_i-\mu}{M}\geq x-s\right)\right]\,\nu_m(\mbox{d} s,\mbox{d} t).
\end{equation}
By the identical distribution of the random environment and the assumption H1),
$$\frac{1}{n}\sum_{i=m+1}^n\left(\frac{X_i-\mu}{M}\right)\overset{\mathcal{D}}{=}\frac{1}{n}\sum_{i=1}^{n-m}\left(\frac{X_i-\mu}{M}\right)\leq \frac{n-m}{n}\quad \P -a.s.$$
Therefore, we can replace $\displaystyle \frac{1}{n}\sum_{i=m+1}^n\left(\frac{X_i-\mu}{M}\right)$ in (\ref{eq6}) by $\displaystyle \frac{1}{n}\sum_{i=1}^{n-m}\left(\frac{X_i-\mu}{M}\right)$ to obtain
\begin{equation}I_1=\int\mathds{1}_{\left\{t\leq\frac{1}{n},\;s\leq \frac{n-m}{n}\right\}}\left[\P\left(\frac{1}{n}\sum_{i=1}^{n-m}\frac{X_i-\mu}{M}\geq x-s-t-\frac{1}{\sqrt{n}}\right)-\P\left(\frac{1}{n}\sum_{i=1}^{n-m}\frac{X_i-\mu}{M}\geq x-s\right)\right]\nu_m(\mbox{d} s,\mbox{d} t).
\end{equation}
Applying Theorem \ref{Fan} to $\displaystyle\left(\frac{1}{n}\sum_{i=1}^{n-m}\frac{X_i-\mu}{M}\right)$, we have for $\displaystyle x\geq3$,

\begin{equation}\label{eqI1_1}
\begin{split}
\P\left(\frac{1}{n}\sum_{i=1}^{n-m}\frac{X_i-\mu}{M}\geq x-s-t-\frac{1}{\sqrt{n}}\right)=& \P\left[\sum_{i=1}^{n-m}\frac{X_i-\mu}{M}\geq n(x-s-t-\frac{1}{\sqrt{n}})\right]\\
\leq& H_{n-m}\left(n\left(x-s-t-\frac{1}{\sqrt{n}}\right), \frac{\sigma\sqrt{n-m}}{M}\right)\\
 \equiv&\;0,
\end{split}
\end{equation}
since if $x\geq3$, $\displaystyle n\left(x-s-t-\frac{1}{\sqrt{n}}\right)\geq n$. And similarly for $x\geq1$,
\begin{equation}\label{eqI1_2}
\begin{split}
\P\left(\frac{1}{n}\sum_{i=1}^{n-m}\frac{X_i-\mu}{M}\geq x-s\right)&= \P\left[\sum_{i=1}^{n-m}\frac{X_i-\mu}{M}\geq n(x-s)\right]\\
                                                                   &\leq H_{n-m}\left(n(x-s),\frac{\sigma\sqrt{n-m}}{M}\right)\\
                                                                   &\equiv0,
\end{split}
\end{equation}
since if $x\geq1$, $n(x-s)\geq n$ . Combining (\ref{eq6}), (\ref{eqI1_1}) and (\ref{eqI1_2}), we hence obtain that for $x\geq 3$,
$$I_1\equiv0.$$

The next step is to find an upper bound for $I_2$. By Markov's inequality and the Lemma 2.4 in \cite{Gra} as below, we obtain that there exist constants $\delta\in(0,1)$ and $c>0$, s.t. for $0\leq m\leq n-1$,
\begin{equation*}
\begin{split}
I_2=&\P\left(|\log W_n-\log W_m|>M\sqrt{n}\right)\\
\leq&\frac{\E\left|\log W_n-\log W_m \right|}{M\sqrt{n}}\\
\leq&\frac{1}{M\sqrt{n}}\sum_{k=m}^{n-1}\E\left|\log W_{k+1}-\log W_k\right|\\
\leq&\frac{c}{M\sqrt{n}}\sum_{k=m}^{n-1}\delta^k=\frac{c\delta^m(1-\delta^{n-m})}{M\sqrt{n}(1-\delta)}\leq\frac{c'\delta^m}{M\sqrt{n}},
\end{split}
\end{equation*}
where $\displaystyle c'=\frac{c}{1-\delta}$.
\end{proof}
\begin{lemma}[\cite{Gra}]
Under assumption H2), there exist constants $c>0$ and $\delta\in(0,1)$, s.t. for any $n\geq0$,
$$\E|\log W_{n+1}-\log W_n|\leq c \delta^n.$$
\end{lemma}

\section{Example: binary branching}\label{secE}\label{Sec3}
A typical example is the binary BP in i.i.d random environment $\overline{\xi}$ with common distribution supported in a finite set. Consider $\nu=\{\nu_k\}_{1\leq k\leq k_0}$ be a probability distribution with support in the finite set $\{a_1,\cdots, a_{k_0}\}\in(0,1)^{k_0}$, where $\nu_k$ is the mass on $a_k$ for each $k$. 

Suppose the random environment of the ($n+1$)th generation $\xi_n$ depends solely on one single random parameter $P_n\in(0,1)$, the success rate of a Bernoulli trial, so that we can denote the entire random environment process by $\overline{\xi}=(P_0,P_1,\cdots)$. Suppose $\{P_0,P_1,\cdots\}$ is a sequence of i.i.d r.v.s following the common probability distribution $\nu$. Given the environment, the offspring distribution in each family is Bernoulli. More precisely, given $\overline{\xi}$, the conditional p.m.f. of the family size $N_{k,i}$ is given by
$$\P_{\overline{\xi}}(N_{k,i}=1)=P_{k},\text{ and }\P(N_{k,i}=2)=1-P_{k},$$
for any $k\geq0$ and $i\geq1$, that is, every family size is either 1 or 2. \\

Since
$$m_0=\E(N_{0,1}|P_0)=1\times P_0+2\times(1-P_0)=2-P_0>1,\quad \nu-a.s.$$
and $p_0(\xi_0)=0$, we have 
$$0<\mu=\nu[\log(2-P_0)]<+\In,$$
$$\E|\log(1-p_0(\xi_0))|=0,$$
$$0<\sigma^2=\nu[\log(2-P_0)-\mu]^2<+\In,$$
$$\E\left(\frac{Z_1\log^+Z_1}{m_0}\right)=\nu\left[\frac{\E_{\xi_0}(Z_1\log^+Z_1)}{m_0}\right]=2(\log2)\,\nu\left[\frac{1-P_0}{2-P_0}\right]<2(\log 2)\nu(1-P_0)<2\log2,$$
$$\text{ for any }i\geq1,\quad X_i=\log(2-P_{i-1})<\log2,\quad\nu-a.s.,$$
$$\E\left(\frac{Z_1}{m_0}\right)^p=\nu\left[\frac{\E_{\xi_0}(Z_1^p)}{m_0^p}\right]=\nu\left[\frac{2^p+(1-2^p)P_0}{(2-P_0)^p}\right]<2^p+(1-2^p)\nu(P_0)<1,\quad \text{for }p>1$$
and
$$\E|\log m_0|^q=\nu[\log(2-P_0)]^q<(\log 2)^q, \quad \text{for }q>2.$$
So the assumptions 1) - 3), H1) and H2) are all fulfilled. From the Theorem \ref{Thm1}, we have there exist constants $C>0$ and $\delta\in(0,1)$ s.t. for any $x\geq3$ and $1\leq m\leq n$,
$$\P\left(\frac{\log Z_n-n\mu}{nM}\geq x\right)\leq \frac{C\delta^m}{M\sqrt{n}},$$
where $\displaystyle \mu=\nu[\log(2-P_0)]$ and $M=\log2-\nu[\log(2-P_0)]$.

\section*{Acknowledgements}
The author is grateful for Dr. Xiequan Fan from Center for Applied Mathematics, Tianjin University, for a couple of discussions on the classical Hoeffding inequality and especially the recent development of this kind of inequalities.


\begin{thebibliography}{99}
\bibitem{Ath1} Athreya, K. B. and Karlin, S. On branching processes with random environments \Rmnum{1}. {\it Ann. Math. Stat.}, {\bf 42 (5)} (1971), 1499-1520.
\bibitem{Ath2} Athreya, K. B. and Karlin, S. On branching processes with random environments \Rmnum{2}. {\it Ann. Math. Stat.}, {\bf 42 (6)} (1971), 1843-1858.
\bibitem{Ben} Bennett, G. Probability inequalities for the sum of independent random variables. {\it J. Amer. Statist. Assoc.}, {\bf 57 (297)} (1962), 33-45.
\bibitem{Boi} Boinghoff, C. and Kersting, G. Upper large deviations of branching processes in a random environment - Offspring distributions with geometrically bounded tails. {\it Stoch. Pro. and Their Appli.}, {\bf 120} (2010), 2064-2077.
\bibitem{Fan} Fan, X., Grama, I., Liu, Q. Hoeffding's inequality for supermartingales. {\it Stoch. Pro. and Their Appli.}, {\bf 122 (10)} (2012), 3545-3559.
\bibitem{Gei1} Geiger, J. and Kersting, G. The survival probability of a critical branching process in random environment, {\it Ther. Verojatnost. \rmnum {1} Primenen}, {\bf 45} (2000), 607-615.
\bibitem{Gei2} Geiger, J., Kersting, G. and Vatutin, V. A. Limit theorems for subcritical branching process in random environment, {\it Ann. I. H. Poincar\'{e}}, {\bf 39 (4)} (2003), 593-620.
\bibitem{Gra} Grama, I., Liu, Q. Miqueu, E. Berry-Esseen's bound and Cram\'{e}r's large deviation expansion for a supercritical branching process in a random environment. {\it Stoch. Pro. and Their Appli.}, {\bf 127 (4)} (2017), 1255-1281.
\bibitem{Gui} Guivarc'h, Y. and Liu, Q.  Asymptotic properties of braching processes in random environment, {\it C. R. Mathematique}, {\bf 332 (4)} (2001), 339-344.
\bibitem{Hoef} Hoeffding, W. Probability inequalities for sums of bounded random variables, {\it J. Amer. Statist. Assoc.}, {bf 58} (1963), 13-30.
\bibitem{Chu} Huang, C., Liu, Q. Moments, moderate and large deviations for a branching process in a random environment. {\it Stoch. Pro. and Their Appli.}, {\bf 122 (2)} (2012), 522-545.
\bibitem{Koz}  Kozlov, M.V. On the asymptotic behavior of the probability of non-extinction for critical branching processes in a random environment. {\it Theory Probab}, {\bf 21(4)} (1976), 791-804.
\bibitem{Emile} Le Page, E., Ye, Y. The survival probability of a critical branching process in a Markovian random environment, {\it C. R. Mathematique}, {\bf 348 (5-6)} (2011), 301-304.
\bibitem{Nak}  Nakashima, M. Lower deviation of branching process in random environment with geometrical offspring distributions. {\it Stoch. Pros. and their Appli.}, {\bf 123} (2013), 3560-3587.
\bibitem{Smi}  Smith, W. L., Wilkinson, W. E., On Branching Processes in Random Environments. {\it Ann. Math. Statist.}, {\bf 40 (3)} (1969), 814-827.
\bibitem{Tan} Tanny, D. A necessary and sufficient condition for a branching process in a random environment to grow like the product of its means. {\it Stoch. Pro. and Their Appli.}, {\bf 28 (1)} (1988), 123-139.

\end{thebibliography}
\end{document}